\documentclass[12pt]{amsart}
\usepackage[all]{xy}
\usepackage{amssymb}

\title{On some triangulated categories over group algebras}
\author{Ioannis Emmanouil and Wei Ren $^*$}

\newtheorem{Lemma}{Lemma}[section]
\newtheorem{Proposition}[Lemma]{Proposition}
\newtheorem{Theorem}[Lemma]{Theorem}
\newtheorem{Corollary}[Lemma]{Corollary}
\newtheorem{Example}[Lemma]{Example}

\newtheorem{Remark}[Lemma]{Remark}

\begin{document}

\begin{abstract}
In this paper, we introduce the cofibrant derived category of a
group algebra $kG$ and study its relation to the derived category
of $kG$. We also define the cofibrant singularity category of $kG$,
whose triviality characterizes the regularity of $kG$ with respect
to the cofibrant dimension, and examine its significance as a measure
of the obstruction to the equality between the classes of Gorenstein
projective and cofibrant modules. We show that the same obstruction
can be measured by certain localization sequences between stable
categories.
\end{abstract}

\makeatletter
\@namedef{subjclassname@2020}{\textup{2020} Mathematics Subject Classification}
\makeatother

\date{}
\thanks{$^*$ Corresponding author; e-mail: wren$\symbol{64}$cqnu.edu.cn}
\maketitle
\tableofcontents

\section{Introduction}

\noindent
Gorenstein homological algebra is the relative homological theory,
which is based upon the classes of Gorenstein projective, Gorenstein
injective and Gorenstein flat modules \cite{EJ, Hol}.
It has developed rapidly during the past several years and has
found interesting applications in the representation theory of
Artin algebras, the theory of singularities and in cohomological
group theory. However, in contrast to classical homological algebra
properties, it is not known whether (i) the class of Gorenstein
projective modules is contravariantly finite (precovering) in the
full module category, (ii) Gorenstein projective modules are
Gorenstein flat and (iii) Gorenstein flat modules are those
modules whose character modules (Pontryagin duals) are Gorenstein
injective. These questions represent fundamental problems in
Gorenstein homological algebra, and pose obstacles in the
development of several aspects of the theory. Alternatively,
in the special case where the ambient ring is a group algebra,
one may consider the cofibrant modules, fibrant modules and
cofibrant-flat modules as substitutes. The notion of cofibrant
modules was introduced by Benson in \cite{Ben} to study modules
of type $FP_\infty$ over groups in the class
${\scriptstyle{{\bf LH}}}\mathfrak{F}$ of hierarchically
decomposable groups introduced by Kropholler \cite{Kro}. The
cofibrant modules are proved to be Gorenstein projective by
an elegant construction of Cornick and Kropholler \cite{CK}
and form a contravariantly finite (precovering) class \cite{ER1}.
Fibrant modules and cofibrant-flat modules are introduced in
\cite{ER2}; these are particular Gorenstein injective and
Gorenstein flat modules respectively. Every cofibrant module
is cofibrant-flat and cofibrant-flat modules are precisely
those modules whose character modules are fibrant. In this
paper, we intend to compare these modules to the Gorenstein
modules from the perspective of the relative singularity and
the stable categories.

We first introduce the cofibrant derived category
${\bf D}_{\tt Cof}(kG)$ of a group algebra $kG$. Using standard
techniques, we can describe the bounded cofibrant derived category
${\bf D}^b_{\tt Cof}(kG)$ as a suitable subcategory of the homotopy
category ${\bf K}(kG)$; cf.\ Theorem \ref{thm:bCDcat}. This description
is reminiscent of \cite[Theorem 3.6(ii)]{GZ}, a result that was only
proved though for finite dimensional algebras. Inspired by
\cite[Theorem 5.1]{KZ}, we also make a comparison with the bounded
derived category ${\bf D}^b(kG)$ and express ${\bf D}^b(kG)$ as a
Verdier quotient of ${\bf D}^b_{\tt Cof}(kG)$ in Theorem
\ref{thm:Dcat-CDcat}.

The singularity category ${\bf D}^b_{sg}(R)$ of a ring $R$ was
introduced by Buchweitz in his unpublished note \cite{Buc}, under
the name of ``stable derived category''. Orlov \cite{Or} called
${\bf D}^b_{sg}(R)$ the singularity category of $R$, since that
category reflects certain homological singularity properties of
$R$. The basic properties of the singularity category
${\bf D}^b_{sg}(R)$ are carefully stated in \cite{Ch, LH}. There
is always a fully faithful functor $F$ from the stable category of
Gorenstein projective modules to ${\bf D}^b_{sg}(R)$. The functor
$F$ is an equivalence if the ring $R$ has finite Gorenstein global
dimension; this result is referred to as Buchweitz's theorem
\cite[4.4.1]{Buc}. A particular case of this result was obtained
by Rickard in \cite[Theorem 2.1]{Ric}, where it was shown that the
singularity category of a self-injective algebra is triangle-equivalent
to its stable module category. The Gorenstein defect category
\cite[Definition 4.1]{BJO} is defined as the Verdier quotient
${\bf D}^b_{sg}(R)/{\rm Im}F$. It follows from
\cite[Theorems 3.6 and 4.2]{BJO} that the converse of Buchweitz's
theorem is also true, so that the functor $F$ is an equivalence
if and only if $R$ has finite Goresntein global dimension; see
also \cite[\S 8.5]{Z}. A description of the Gorenstein defect
category is given in \cite[Theorem 6.7]{KZ}. The Gorenstein
singularity category was introduced in \cite{BDZ}. It is
equivalent to the Gorenstein defect category when the class
of Gorenstein projective modules is contravariantly finite;
cf.\ \cite[Theorem 4.3]{BDZ} or \cite[Theorem 6.7(ii)]{KZ}.

In section 4, we introduce the cofibrant singularity category
${\bf D}^b_{{\tt Cof}.sg}(kG)$; its triviality is equivalent
to the regularity of $kG$ with respect to the cofibrant dimension.
We show that $kG$ is Gorenstein if and only if the cofibrant
singularity category is triangle-equivalent to the quotient
$\underline{{\tt GProj}}(kG) / \underline{{\tt Cof}}(kG)$ of
the stable categories of Gorenstein projective and cofibrant
$kG$-modules; see Theorem \ref{thm:Cof-sg2}(ii). The cofibrant
singularity category measures the obstruction to the equivalence
between the stable categories of Gorenstein projective and cofibrant
modules (equivalently, to the equality between the categories of
Gorenstein projective and cofibrant modules), in the case where
$kG$ is Gorenstein; see Theorem \ref{thm:Cof-sg=0}. If $G$ is a
group of type $FP_\infty$ contained in the class
${\scriptstyle{{\bf LH}}}\mathfrak{F}$ introduced by Kropholler
\cite{Kro}, then the cofibrant singularity category of the
integral group algebra $\mathbb{Z}G$ is trivial. However, if $G$
is any finite group and $k = \mathbb{Z}/(p^2)$, where $p$ is a
prime number, then the cofibrant singularity category of $kG$ is
non-trivial, since there are Gorenstein projective $kG$-modules
which are not cofibrant; see Example \ref{eg:2}. We prove a version
of the converse of Buchweitz's theorem regarding cofibrant modules
in Proposition \ref{prop:dense}: If the canonical functor $F$ from
the stable category of cofibrant modules to the singularity category
is an equivalence, then all $kG$-modules have finite cofibrant dimension.
This is analogous to the corresponding result for Gorenstein projective
modules. However, we give a direct and concise argument which can be
also used to prove the converse of Buchweitz's theorem regarding
Gorenstein projective modules.

In section 5 we obtain certain localization sequences that involve
stable categories and use them to compare the class of projectively
coresolved Gorenstein flat modules \cite{SS} to the class of cofibrant
modules (resp.\ the class of Gorenstein flat modules to the class of
cofibrant-flat modules, resp.\ the class of Gorenstein injective
modules to the class of fibrant modules). We note that the stable
category of cofibrant modules is equivalent for many groups $G$ to
the stable module category ${\tt StMod}(kG)$, which is itself defined
in \cite[Sections 8-10]{Ben} and has objects the modules of finite
cofibrant dimension; see \cite[Corollary 6.9]{R}. Finally, in Proposition
\ref{prop:perf-ring}, we obtain a characterization of perfect group
algebras in terms of the behaviour of cofibrant and Gorenstein modules,
in a way that is analogous to Bass' characterization of perfect rings
in terms of flat modules \cite{Bass}.

\vspace{0.1in}

\noindent
{\em Notations and terminology.}
We consider a commutative ring $k$ and a group $G$. Unless
otherwise specified, all modules are $kG$-modules.

\section{Modules over group algebras}

\noindent
Let $M,N$ be two $kG$-modules. Using the diagonal action of $G$,
the tensor product $M \otimes_k N$ is also a $kG$-module;
we let $g \cdot (x \otimes y) = gx \otimes gy \in M \otimes_k N$
for all $g \in G$, $x \in M$, $y \in N$. The $k$-module
${\rm Hom}_k(M, N)$ admits the structure of a $kG$-module as well
with the group $G$ acting diagonally; here
$(g\cdot f)(x)= gf(g^{-1}x)\in N$ for any $g\in G$,
$f\in {\rm Hom}_k(M, N)$ and $x\in M$.

Let $B(G,\mathbb{Z})$ be the $\mathbb{Z}G$-module consisting
of all bounded functions from $G$ to $\mathbb{Z}$, introduced
in \cite{KT}. The $kG$-module
$B(G,k) = B(G,\mathbb{Z}) \otimes_{\mathbb{Z}}k$ is identified
with the module of all functions from $G$ to $k$ that admit
finitely many values. It is free as a $kH$-module for any
finite subgroup $H \subseteq G$. We also note that there
is a $k$-split $kG$-linear monomorphism
$\iota: k \rightarrow B(G,k)$. For simplicity of notation,
we shall denote $B(G,k)$ by $B$. Following Benson \cite{Ben}, we
call a $kG$-module $M$ cofibrant if the (diagonal) $kG$-module
$M \otimes_k B$ is projective; all projective $kG$-modules are
cofibrant. Cofibrant modules are closely related to Gorenstein
projective modules; we refer to \cite{EJ, Hol} for the notion
of Gorenstein projective module and to \cite[$\S 4$]{SS} for
the particular subclass of projectively coresolved Gorenstein
flat modules. Let ${\tt GProj}(kG)$ and ${\tt PGF}(kG)$ denote
the classes of Gorenstein projective and projectively coresolved
Gorenstein flat $kG$-modules, respectively. Using an elegant
construction, it was shown in \cite[Theorem 3.5]{CK} that any
cofibrant module is Gorenstein projective. Analogously, it was
proved in \cite[Proposition 8.2]{S} that any cofibrant module
is projectively coresolved Gorenstein flat. Hence,
${\tt Cof}(kG) \subseteq {\tt PGF}(kG) \subseteq {\tt GProj}(kG)$.

Recall that a group $G$ is of type $\Phi$ over $k$ if for any
$kG$-module $M$, the projective dimension of $M$ is finite if
and only if for any finite subgroup $H\subseteq G$, the
restricted $kH$-module $\mbox{res}_H^GM$ has finite projective
dimension. The concept of groups of type $\Phi$ over $\mathbb{Z}$
was introduced by Talelli \cite{T}, in order to study groups
which admit a finite dimensional model for the classifying
space of proper actions. We also consider the class of
hierarchically decomposable groups defined by Kropholler
\cite{Kro}. The class ${\scriptstyle{\bf H}}\mathfrak{F}$
is the smallest class of groups, which contains the class
$\mathfrak{F}$ of finite groups and is such that whenever
a group $G$ admits a finite dimensional contractible
$G$-CW-complex with stabilizers in
${\scriptstyle{\bf H}}\mathfrak{F}$, then we also have
$G \in {\scriptstyle{\bf H}}\mathfrak{F}$. The class
${\scriptstyle{\bf LH}}\mathfrak{F}$ consists of those
groups, all of whose finitely generated subgroups are in
${\scriptstyle{\bf H}}\mathfrak{F}$. This class contains
all soluble-by-finite groups, all groups of finite
cohomological dimension over $\mathbb{Z}$ and all groups
admitting a faithful representation as endomorphisms of a
Noetherian module over a commutative ring.

If $k$ has finite global dimension and $G$ is a either an
${\scriptstyle{{\bf LH}}}\mathfrak{F}$-group or a group of
type $\Phi$ over $k$, then every Gorenstein projective
$kG$-module is cofibrant and hence
${\tt Cof}(kG)={\tt PGF}(kG)={\tt GProj}(kG)$; cf. \cite{DT, Bi}.
The assumption on $k$ is relaxed in \cite[Corollary 2.5]{ER1}, where
it is shown that the above equality also holds if (i) $k$ has finite
weak global dimension or (ii) all pure-projective $k$-modules have
finite projective dimension. It was conjectured in
\cite[Conjecture 1.1]{Bi} that for any group $G$ and any commutative
ring $k$ of finite global dimension, the class of Gorenstein projective
$kG$-modules coincides with the class of cofibrant $kG$-modules, i.e.\
that ${\tt Cof}(kG)= {\tt GProj}(kG)$. This conjecture was initially
proposed by Dembegioti and Talelli over $\mathbb{Z}$ in \cite{DT} and
is still open.

We recall that the group algebra $kG$ is {\em Gorenstein}
if every $kG$-module has finite Gorenstein projective dimension.

\begin{Lemma}\label{lem:bound-cd}
The following conditions are equivalent:

(i) Every $kG$-module has finite cofibrant dimension.

(ii) The group algebra $kG$ is Gorenstein and ${\tt Cof}(kG) = {\tt GProj}(kG)$.
\end{Lemma}

\begin{proof}
The implication (ii)$\rightarrow$(i) is obvious. Since any cofibrant
module is Gorenstein projective, assertion (i) implies that the
Gorenstein projective dimension of any $kG$-module $M$ is finite;
then, the group algebra $kG$ is Gorenstein. Since the cofibrant
dimension of any (Gorenstein projective) $kG$-module is finite,
\cite[Proposition 2.2(i)]{ER1} implies that that any Gorenstein
projective $kG$-module is cofibrant, i.e.\ that
${\tt Cof}(kG) = {\tt GProj}(kG)$.
\end{proof}

\noindent
Over a general group algebra, there may be Gorenstein projective
(even projectively coresolved Gorenstein flat) modules which are
not cofibrant.

\begin{Example}\label{eg:1}
Let $G$ be a finite group and $p$ a prime number. We consider
the local ring $k = \mathbb{Z}/(p^2)$ and its maximal ideal
$\mathfrak{m} = (p)/(p^2)$. Note that the ring $k$ is self-injective
and the group algebra $kG$ is Gorenstein.
The $k$-module $M'= \mathbb{Z}/(p) \cong k/\mathfrak{m}$ is
Gorenstein projective and has infinite projective dimension.
It follows that the induced $kG$-module $M = kG \otimes_k M'$
is Gorenstein projective; see, for example, \cite[Lemma 2.6(i)]{ET1}.
In fact, $M$ is also projectively coresolved Gorenstein flat.
However, $M$ is not cofibrant. Indeed, the projectivity of the diagonal
$kG$-module $M \otimes_k B$ would imply that the underlying $k$-module
$M \otimes_k B$ is projective. It would then follow that $M$ is
$k$-projective and hence that the $k$-module $M'$ is also projective.
\end{Example}

\noindent
For any $kG$-module $M$, the cofibrant dimension of $M$ is precisely
the projective dimension of $M\otimes_k B$.
Recall that the Gorenstein cohomological dimension ${\rm Gcd}_k G$
of a group $G$ is defined as the Gorenstein projective dimension of
the trivial $kG$-module $k$. For a commutative ring $k$ of finite
global dimension, it follows from \cite[Theorem 1.7]{ET1} that
${\rm Gcd}_k G < \infty$ if and only if the group algebra $kG$ is
Gorenstein. Moreover, it is easily seen that
${\rm proj.dim}_{kG} B <\infty$ if and only if
${\rm proj.dim}_{kG}M\otimes_k B <\infty$ for any $kG$-module $M$
which has finite projective dimension as a $k$-module.

\begin{Proposition}
Let $k$ be a ring of finite global dimension and consider a group $G$.

(i) If ${\rm proj.dim}_{kG} B <\infty$, then ${\rm Gcd}_k G < \infty$
and ${\tt Cof}(kG) = {\tt GProj}(kG)$.

(ii) If ${\tt Cof}(kG) = {\tt GProj}(kG)$, then
${\rm proj.dim}_{kG} B = {\rm Gcd}_k G$.
\end{Proposition}

\begin{proof}
(i) If ${\rm proj.dim}_{kG} B$ is finite, it follows from
\cite[Theorem 1.7]{ET1} that $B$ is a characteristic module
and ${\rm Gcd}_k G < \infty$. Since any $kG$-module $M$ has
finite projective dimension as a $k$-module, the finiteness
of ${\rm proj.dim}_{kG} B$ implies that
${\rm proj.dim}_{kG} M\otimes_k B <\infty$, so that $M$ has
finite cofibrant dimension. Hence, Lemma \ref{lem:bound-cd}
implies that ${\tt Cof}(kG) = {\tt GProj}(kG)$.

(ii) If ${\tt Cof}(kG) = {\tt GProj}(kG)$, then the Gorenstein
cohomological dimension ${\rm Gcd}_k G$ of $G$ is equal to the
cofibrant dimension of the trivial $kG$-module $k$, i.e.\ to
the projective dimension ${\rm proj.dim}_{kG} B$ of $B$
\end{proof}

\noindent
A $kG$-module $M$ is fibrant \cite{ER2} if the diagonal
$kG$-module ${\rm Hom}_k(B,M)$ is injective. Fibrant modules
are closely related to Gorenstein injective modules; we refer
to \cite{Hol, EJ} for this notion. Using the dual version of
\cite[Theorem 3.5]{CK}, it is shown in \cite[Proposition 5.6(i)]{ET2}
that any fibrant module is Gorenstein injective. Hence, the
class ${\tt Fib}(kG)$ of fibrant modules is always contained
in the class ${\tt GInj}(kG)$ of Gorenstein injective modules.
A $kG$-module $M$ is cofibrant-flat \cite{ER2} if the (diagonal)
$kG$-module $M \otimes_k B$ is flat; we shall denote by
${\tt Cof.flat}(kG)$ the class of cofibrant-flat modules. These
modules are closely related to Gorenstein flat modules. We refer
to \cite{EJ, Hol} for this notion and let ${\tt GFlat}(kG)$ denote
the class of Gorenstein flat $kG$-modules. Using the homological
version of \cite[Theorem 3.5]{CK}, it is shown in
\cite[Proposition 5.2(i)]{ET2} that any cofibrant-flat module is
Gorenstein flat and hence ${\tt Cof.flat}(kG) \subseteq {\tt GFlat}(kG)$.
If $k$ has finite global dimension and $G$ is a group that is either
contained in the class ${\scriptstyle{{\bf LH}}}\mathfrak{F}$ or else
has type $\Phi$, then \cite[Proposition 4.3]{ER2} (resp.\
\cite[Proposition 2.3]{ER2}) implies that
${\tt Fib}(kG)={\tt GInj}(kG)$ (resp. ${\tt Cof.flat}(kG)= {\tt GFlat}(kG)$).
In analogy to the conjecture by Dembegioti and Talelli \cite{DT}
about cofibrant modules, we may ask whether these equalities are
always true.

\section{The cofibrant derived category}

\noindent
In this section we introduce a relative derived category,
the cofibrant derived category, describe the bounded cofibrant
derived category and express the bounded derived category as
a Verdier quotient of the bounded cofibrant derived category.
These results will be useful in studying (in the next section)
the cofibrant singularity category.

If $R$ is a ring and $\mathcal{A}$ is a class of $R$-modules,
then the relative derived category ${\bf D}_{\mathcal{A}}(R)$
with respect to $\mathcal{A}$ is obtained by formally inverting
all $\mathcal{A}$-quasi-isomorphisms. Here, a morphism of complexes
$f: X\rightarrow Y$ is called an $\mathcal{A}$-quasi-isomorphism,
if the cochain map
\[ {\rm Hom}_R(A,f): {\rm Hom}_R(A, X) \rightarrow {\rm Hom}_R(A,Y) \]
is a quasi-isomorphism for all $A\in \mathcal{A}$. We say that
a complex of $R$-modules $X$ is $\mathcal{A}$-acyclic, if the
complex of abelian groups ${\rm Hom}_R(A, X)$ is acyclic for
all $A\in \mathcal{A}$. It follows that $f$ is an
$\mathcal{A}$-quasi-isomorphism if and only if its mapping cone
${\rm Con}(f)$ is $\mathcal{A}$-acyclic; cf.\
\cite[Chapter 5, Theorem 1.10.2]{GM}. Let
${\bf K}_{\mathcal{A}\text{-ac}}(R) \subseteq {\bf K}(R)$ be the
full triangulated subcategory of the homotopy category, which
consists of all $\mathcal{A}$-acyclic complexes. Then, the
relative derived category ${\bf D}_{\mathcal{A}}(R)$ is precisely
the Verdier quotient of ${\bf K}(R)$ modulo the subcategory
${\bf K}_{\mathcal{A}\text{-ac}}(R)$. We recall that the subcategory
$\mathcal{A}$ is said to be admissible, if it is contravariantly finite
(so that any module has a right $\mathcal{A}$-approximation) and
each right $\mathcal{A}$-approximation is surjective; cf.\ \cite{Chen}.
In that case, the relative derived category ${\bf D}_{\mathcal{A}}(R)$
coincides with Neeman's derived category of the exact category
$(R\text{-Mod},\mathcal{E}_{\mathcal{A}})$, where
$\mathcal{E}_{\mathcal{A}}$ is the class of $\mathcal{A}$-acyclic
short exact sequences; cf.\ \cite[Construction 1.5]{N1} or \cite{Kel}.
We also note that if $\mathcal{A}$ is the class of Gorenstein
projective modules, then ${\bf D}_{\mathcal{A}}(R)$ is the
Gorenstein derived category, in the sense of Gao and Zhang \cite{GZ}.

\vspace{0.1in}

\noindent
{\sc I.\ Inverting cofibrant-quasi-isomorphisms.}
We say that an acyclic complex $X$ of $kG$-modules is
cofibrant-acyclic if the complex ${\rm Hom}_{kG}(C,X)$
is acyclic for any cofibrant module $C$.
To simplify the notation, we denote the homotopy category
${\bf K}_{{\tt Cof(kG)}\text{-ac}}(kG)$ of cofibrant-acyclic
(${\tt Cof}(kG)$-acyclic) complexes over the group algebra $kG$ by
${\bf K}_{cac}(kG)$. There is an important characterization of thick
subcategories, due to Rickard: a full triangulated subcategory
$\mathcal{K}$ of a triangulated  category $\mathcal{T}$ is thick
if and only if every direct summand of an object of $\mathcal{K}$
is in $\mathcal{K}$; cf.\ \cite[Proposition 1.3]{Ric} or
\cite[Criterion 1.3]{N1}. By Rickard's criterion, it follows
immediately that for $* \in \{{blank, -, +, b}\}$,
${\bf K}^{*}_{cac}(kG)$ is a thick subcategory of ${\bf K}^{*}(kG)$.

The {\em cofibrant derived category} ${\bf D}^{*}_{\tt Cof}(kG)$ is
defined to be the Verdier quotient of ${\bf K}^{*}(kG)$ modulo the
thick subcategory  ${\bf K}^{*}_{cac}(kG)$, i.e.
\[ {\bf D}^{*}_{\tt Cof}(kG): = {\bf K}^{*}(kG) / {\bf K}^{*}_{cac}(kG)
= S^{-1}{\bf K}^{*}(kG) . \]
Here, $S$ is the compatible multiplicative system of morphisms
determined by ${\bf K}^{*}_{cac}(kG)$, i.e.\ the collection of
all ${\tt Cof}(kG)$-quasi-isomorphisms in ${\bf K}^{*}(kG)$. As
an immediate consequence of \cite[Proposition 2.6]{CFH}, it
follows that a morphism $f: X\rightarrow Y$ in ${\bf K}(kG)$ is
a ${\tt Cof}(kG)$-quasi-isomorphism if and only if it induces an
isomorphism
\[ {\rm Hom}_{{\bf K}(kG)}(M, X[n]) \longrightarrow
   {\rm Hom}_{{\bf K}(kG)}(M, Y[n])\]
for any complex $M\in {\bf K}^{-}({\tt Cof}(kG))$ and any integer $n$.
Consequently, by a standard argument we have the following result, which
implies that the functor
$F:{\bf K}^{*}({\tt Cof}(kG)) \rightarrow {\bf D}^{*}_{\tt Cof}(kG)$
obtained as the composition of the embedding
${\bf K}^{*}({\tt Cof}(kG)) \rightarrow {\bf K}^{*}(kG)$ followed by
the localization functor $Q: {\bf K}^{*}(kG)\rightarrow {\bf D}^{*}_{\tt Cof}(kG)$
is fully faithful for $*\in \{b, -\}$.

\begin{Lemma}\label{lem:iso1}
Let $M$ belong to ${\bf K}^{-}({\tt Cof}(kG))$ and $X$ be an
arbitrary complex of $kG$-modules. Then, the canonical map
$f \mapsto f/{\rm Id}_{M}$ is an isomorphism of abelian groups
\[ \varphi : {\rm Hom}_{{\bf K}(kG)}(M, X) \longrightarrow
{\rm Hom}_{{\bf D}_{{\tt Cof}}(kG)}(M, X) . \]
\end{Lemma}

\noindent
The following result is also standard. We note that assertion
(i) analogous to \cite[Proposition 2.7]{GZ}, where the Gorenstein
derived category is considered.

\begin{Proposition}\label{prop:subDcat}
(i) ${\bf D}^{b}_{\tt Cof}(kG)$ and ${\bf D}^{-}_{\tt Cof}(kG)$ are
triangulated subcategories of ${\bf D}_{\tt Cof}(kG)$.

(ii) For $*\in \{b, -\}$, ${\bf K}^{*}({\tt Cof}(kG))$ can be viewed
as a triangulated subcategory of ${\bf D}^{*}_{\tt Cof}(kG)$.
\end{Proposition}

\noindent
In order to characterize the bounded cofibrant derived category, we
consider the following subcategory of ${\bf K}^{-}({\tt Cof}(kG))$:
\[ {\bf K}^{-,cab}({\tt Cof}(kG)):= \left\{X\in {\bf K}^{-}({\tt Cof}(kG))
   \, \vline \,\begin{matrix}  \text{there exists }
   n=n(X)\in\mathbb{Z}, \text{ such that }\\
   {\rm H}^{i}{\rm Hom}_{kG}(M, X) = 0, \forall i\leq n,
   \forall M\in {\tt Cof}(kG)  \end{matrix} \right\},\]
where the superscript ``cab'' stands for ``${\tt Cof}(kG)$-acyclic bounded''.

\begin{Lemma}\label{lem:res}
There exists a functor
$\Theta: {\bf K}^{b}(kG) \rightarrow {\bf K}^{-,cab}({\tt Cof}(kG))$,
such that for each $X\in {\bf K}^{b}(kG)$ there is a functorial
${\tt Cof}(kG)$-quasi-isomorphism
$\theta_X: \Theta(X)\rightarrow X$.
\end{Lemma}

\begin{proof}
Let $X$ be a complex in ${\bf K}^{b}(kG)$ and denote by $w(X)$ its width,
i.e.\ the number of its non-zero components. If $w(X) = 1$, then $X$ is a
complex consisting of a single module $N$ in some degree and zeroes elsewhere.
It follows from \cite[Theorem 3.3]{ER1} that $({\tt Cof}(kG), {\tt Cof}(kG)^\perp)$
is a complete cotorsion pair. Hence, there is a short exact sequence
\[ 0\rightarrow K^{-1}\rightarrow C^0\rightarrow N\rightarrow 0 . \]
where $C^0 \in {\tt Cof}(kG)$ and $K^{-1} \in {\tt Cof}(kG)^\perp$, so that
$C^0 \rightarrow N$ is a special right ${\tt Cof}(kG)$-approximation of $N$.
Considering successive special right ${\tt Cof}(kG)$-approximations,
we obtain an acyclic complex
\[ \cdots \rightarrow C^{-1}\rightarrow C^0
   \stackrel{p}{\rightarrow} N\rightarrow 0 . \]
Let $\Theta(X)$ be the suitably shifted and deleted complex, so that
the linear map $p$ defines a ${\tt Cof}$-quasi-isomorphism
$\theta_X : \Theta(X)\rightarrow X$. Inductively,
assume that $w(X)\geq 2$, and there is an integer $n$
such that $X^n\neq 0$ and $X^i = 0$ for any $i<n$. Let $X' = X^{\geq n+1}$ be
the left brutal truncation of $X$ at $n+1$, and $X'' = X^n[-n]$ be the complex
with $X^n$ concentrated in degree $n$. There is a map
$f: X''[-1]\rightarrow X'$, which induces a distinguished triangle
\[ X''[-1] \stackrel{f}\longrightarrow X'\longrightarrow X\longrightarrow X'' \]
in ${\bf K}^{-}(kG)$. By induction, we have the following diagram
\[ \xymatrix{ \Theta(X''[-1]) \ar[r]^{\Theta(f)} \ar[d]_{\theta_{X''[-1]}}
& \Theta(X') \ar[r] \ar[d]_{\theta_{X'}}
& {\rm Con}(\Theta(f)) \ar[r]\ar@{-->}[d]_{\theta_X} & \Theta(X''[-1])[1]\\
X''[-1] \ar[r]^{f} & X' \ar[r] & X \ar[r] &X'' }\]
Let $\Theta(X) = {\rm Con}(\Theta(f))$. Then,
$\Theta(X) \in {\bf K}^{-,cab}({\tt Cof}(kG))$ and there exists
a ${\tt Cof}(kG)$-quasi-isomorphism $\theta_X: \Theta(X)\rightarrow X$.
It is standard that the $\theta_X$'s are functorial.
\end{proof}

\noindent
We can now describe the bounded cofibrant derived category; cf.\
\cite[Theorem 3.6(ii)]{GZ} for the case of the bounded Gorenstein
derived category of a finite dimensional algebra. The result
follows by considering the functor
$F: {\bf K}^{-,cab}({\tt Cof}(kG)) \rightarrow
    {\bf D}^{-}_{\tt Cof}(kG)$, which is fully faithful and dense
by invoking Proposition \ref{prop:subDcat} and Lemma \ref{lem:res}.

\begin{Theorem}\label{thm:bCDcat}
There is a triangle-equivalence
${\bf D}^{b}_{\tt Cof}(kG)\simeq {\bf K}^{-,cab}({\tt Cof}(kG))$.
\end{Theorem}

\begin{Remark}
We say that a morphism of complexes $f: X\rightarrow Y$ is
a fibrant-quasi-isomorphism if the cochain map
${\rm Hom}_{kG}(f, M): {\rm Hom}_{kG}(Y, M)\rightarrow {\rm Hom}_{kG}(X, M)$
is a quasi-isomorphism for any fibrant module $M$. Since
injective $kG$-modules are fibrant, every fibrant-quasi-isomorphism
is a quasi-isomorphism and its mapping cone is acyclic. Hence, a
morphism of complexes $f$ is a fibrant-quasi-isomorphism if and
only if the mapping cone ${\rm Con}(f)$ remains acyclic after applying
the functor ${\rm Hom}_{kG}(\_\!\_ , M)$ for any fibrant $kG$-module
$M$. Dually to the above, one can define the fibrant derived category
${\bf D}^{*}_{\tt Fib}(kG)$ to be the localization of ${\bf K}^{*}(kG)$
with respect to fibrant-quasi-isomorphisms.
\end{Remark}

\vspace{0.05in}

\noindent
{\sc II.\ Comparison with the bounded derived category.}
We shall conclude this section by comparing the bounded cofibrant
derived category ${\bf D}^{b}_{\tt Cof}(kG)$ and the bounded derived
category ${\bf D}^{b}(kG)$. We need the following result; the proof
is analogous to \cite[Lemma 5.2(ii)]{KZ}.

\begin{Lemma}\label{lem:b+ac}
Let $C\in {\bf K}^{-,cab}({\tt Cof}(kG))$. If $C$ is acyclic, then
$C \in {\bf K}^{b}_{ac}({\tt Cof}(kG))$.
\end{Lemma}

We shall use the following result from \cite[Corollaire 4-3]{V} repeatedly:
Let $\mathcal{T}_1$, $\mathcal{T}_2$ be triangulated subcategories
of a triangulated category $\mathcal{T}$ with
$\mathcal{T}_2 \subseteq \mathcal{T}_1$. Then, there is an isomorphism
of triangulated categories
$(\mathcal{T}/\mathcal{T}_2)/ (\mathcal{T}_1 / \mathcal{T}_2) \simeq
 \mathcal{T}/ \mathcal{T}_1$.
Since ${\bf D}^{*}_{\tt Cof}(kG) = {\bf K}^{*}(kG)/{\bf K}^{*}_{cac}(kG)$
and ${\bf D}^{*}(kG) = {\bf K}^{*}(kG) / {\bf K}^{*}_{ac}(kG)$, there
is an equivalence of triangulated categories
\[ {\bf D}^{*}(kG) \simeq {\bf D}^{*}_{\tt Cof}(kG) /
   ({\bf K}^{*}_{ac}(kG)/{\bf K}^{*}_{cac}(kG)).\]

\begin{Theorem}\label{thm:Dcat-CDcat}
There are triangle-equivalences
\[ \mathbf{D}^{b}(kG)\simeq\mathbf{D}_{\tt Cof}^{b}(kG) /
   {\bf K}^{b}_{ac}({\tt Cof}(kG)) \simeq
   {\bf K}^{-,cab}({\tt Cof}(kG))/{\bf K}^{b}_{ac}({\tt Cof}(kG)). \]
\end{Theorem}

\begin{proof}
The second equivalence is an immediate consequence of Theorem
\ref{thm:bCDcat}. Regarding the first equivalence, it suffices
to prove that
${\bf K}^{b}_{ac}({\tt Cof}(kG)) \simeq
 {\bf K}^{b}_{ac}(kG)/{\bf K}^{b}_{cac}(kG)$.
We restrict
$F : {\bf K}^{b}({\tt Cof}(kG)) \rightarrow
     {\bf K}^{b}(kG)/{\bf K}^{b}_{cac}(kG) =
     \mathbf{D}^{b}_{\tt Cof}(kG)$
to ${\bf K}^{b}_{ac}({\tt Cof}(kG))$ and obtain a functor
$F' : {\bf K}^{b}_{ac}({\tt Cof}(kG)) \rightarrow
      {\bf K}^{b}_{ac}(kG)/{\bf K}^{b}_{cac}(kG)$.
Since ${\bf K}^{b}_{ac}(kG)/{\bf K}^{b}_{cac}(kG)$ is a subcategory of
${\bf D}^{b}_{\tt Cof}(kG) = {\bf K}^{b}(kG)/{\bf K}^{b}_{cac}(kG)$,
Lemma \ref{lem:iso1} implies that the functor $F'$ is fully faithful.
For any $X \in {\bf K}^{b}_{ac}(kG)$, Lemma \ref{lem:res} implies the
existence of a ${\tt Cof}(kG)$-quasi-isomorphism
$C \rightarrow X$ with $C \in {\bf K}^{-,cab}({\tt Cof}(kG))$.
Since $X$ is acyclic, $C$ is also acyclic; then, Lemma \ref{lem:b+ac}
implies that $C \in {\bf K}^{b}_{ac}({\tt Cof}(kG))$. It follows that
$X \cong F'(C)$ in $\mathbf{D}^{b}_{\tt Cof}(kG)$ and hence $F'$ has
dense image. This completes the proof.
\end{proof}

\section{The cofibrant singularity category}

\noindent
In this section, we intend to study the cofibrant singularity category and
examine its relation to the singularity category and the stable category of
cofibrant modules.

For any abelian category $\mathcal{A}$ with enough projective objects, the
singularity category \cite{Buc, Or} is defined as the Verdier quotient
\[ {\bf D}^b_{sg}(\mathcal{A}):= {\bf D}^b(\mathcal{A}) /{\bf K}^b(\mathcal{P}(\mathcal{A}))
= {\bf K}^{-, b}(\mathcal{P}(\mathcal{A}))/{\bf K}^b(\mathcal{P}(\mathcal{A})). \]
We note that ${\bf D}^b_{sg}(\mathcal{A})=0$ if and only if $\mathcal{A}$ has finite
global dimension. There is a canonical functor
$F: \underline{{\tt GProj}}(\mathcal{A})\rightarrow
 {\bf D}^b_{sg}(\mathcal{A})$,
which sends every Gorenstein projective object in $\mathcal{A}$ to the corresponding
complex concentrated in degree zero. It follows from \cite[Theorem 4.4.1]{Buc}
and \cite[Theorem 2.1]{Ch} that the triangulated functor $F$ is fully faithful.
Following Bergh, J{\o}rgensen and Oppermann \cite{BJO}, the Gorenstein defect
category of $\mathcal{A}$ is defined to be the Verdier quotient
${\bf D}^b_{def}(\mathcal{A}):={\bf D}^b_{sg}(\mathcal{A})/{\rm Im}F$, where
${\rm Im}F \simeq \underline{{\tt GProj}}(\mathcal{A})$ is a thick subcategory of
${\bf D}^b_{sg}(\mathcal{A})$. It follows from \cite{BJO} that
${\bf D}^b_{def}(\mathcal{A})$ is trivial if and only if each object in $\mathcal{A}$
has finite Gorenstein projective dimension. Analogously, the Gorenstein singularity
category of $\mathcal{A}$ is defined as the Verdier quotient of the bounded Gorenstein
derived category modulo the homotopy category of bounded complexes of Gorenstein
projective objects; cf.\ \cite[Definition 4.1]{BDZ}. If the abelian category
$\mathcal{A}$ is CM-contravariantly finite, then \cite[Theorem 4.3]{BDZ} or
\cite[Theorem 6.7(ii)]{KZ} implies
that the Gorenstein singularity category is triangle-equivalent to the Gorenstein
defect category.

\vspace{0.1in}

\noindent
{\sc I.\ Description of the cofibrant singularity category.}
By Theorem \ref{thm:bCDcat},
${\bf D}^b_{\tt Cof}(kG) = {\bf K}^{-, cab}({\tt Cof}(kG))$. The
{\em cofibrant singularity category} is defined as the Verdier quotient
\[ {\bf D}^b_{{\tt Cof}.sg}(kG) =
   {\bf K}^{-, cab}({\tt Cof}(kG))/{\bf K}^b({\tt Cof}(kG)) =
   {\bf D}^b_{\tt Cof}(kG) /{\bf K}^b({\tt Cof}(kG)) . \]
Let ${\bf D}^{b}(kG)_{f{\tt Cof}}$ denote the full subcategory of the
bounded derived category ${\bf D}^b(kG)$ formed by those complexes
isomorphic to a bounded complex of cofibrant modules.

\begin{Lemma}\label{lem:tri-eq1}
There is a triangle-equivalence
${\bf D}^{b}(kG)_{f{\tt Cof}} \simeq
 {\bf K}^b({\tt Cof}(kG)) / {\bf K}^b_{ac}({\tt Cof}(kG))$.
\end{Lemma}

\begin{proof}
Consider the functor
$F : {\bf K}^{-,cab}({\tt Cof}(kG)) \rightarrow \mathbf{D}^{-}(kG)$,
which is trivial on
${\bf K}^{b}_{ac}({\tt Cof}(kG))$, and then induces a triangulated functor
$\overline{F}: {\bf K}^{-,cab}({\tt Cof}(kG)) /
 {\bf K}^{b}_{ac}({\tt Cof}(kG)) \rightarrow
 \mathbf{D}^{-}(kG)$.
The image of $\overline{F}$ is contained in $\mathbf{D}^{b}(kG)$;
in fact, Theorem \ref{thm:Dcat-CDcat} implies that there is an equivalence
$\overline{F} : {\bf K}^{-,cab}({\tt Cof}(kG)) /
 {\bf K}^{b}_{ac}({\tt Cof}(kG)) \rightarrow \mathbf{D}^{b}(kG)$.
 Restricting $\overline{F}$, we obtain an
equivalence
$\overline{F}: {\bf K}^{b}({\tt Cof}(kG)) /
 {\bf K}^{b}_{ac}({\tt Cof}(kG)) \rightarrow
 \mathbf{D}^{b}(kG)_{f{\tt Cof}}$,
as needed.
\end{proof}

\begin{Proposition}\label{prop:Cof-sg1}
There is a triangle-equivalence
${\bf D}^b_{{\tt Cof}.sg}(kG) \simeq
 {\bf D}^{b}(kG) / {\bf D}^{b}(kG)_{f{\tt Cof}}$.
\end{Proposition}

\begin{proof}
Consider the following commutative diagram of triangulated categories
\[ \xymatrix{
{\bf K}^{b}({\tt Cof}) / {\bf K}^{b}_{ac}({\tt Cof}) \ar[d]_{\simeq}
\ar@{^{(}->}[r] & {\bf K}^{-,cab}({\tt Cof}) / {\bf K}^{b}_{ac}({\tt Cof})
\ar[r] \ar[d]_{\simeq} & {\bf K}^{-,cab}({\tt Cof}) / {\bf K}^{b}({\tt Cof})
\ar@{.>}[d] \\ {\bf D}^b(kG)_{f{\tt Cof}} \ar@{^{(}->}[r] & {\bf D}^b(kG)
\ar[r] & {\bf D}^b(kG) / {\bf D}^b(kG)_{f{\tt Cof}}} \]
where the first vertical equivalence comes from Lemma \ref{lem:tri-eq1} and
the second one from Theorem \ref{thm:Dcat-CDcat}. Then, the existence of an
equivalence
\[ {\bf D}^b_{{\tt Cof}.sg}(kG) =
   {\bf K}^{-, cab}({\tt Cof}(kG)) / {\bf K}^b({\tt Cof}(kG)) \simeq
   {\bf D}^{b}(kG) / {\bf D}^{b}(kG)_{f{\tt Cof}} \]
follows readily from the commutative diagram.
\end{proof}

\vspace{0.1in}

\noindent
{\sc II.\ Comparison with the singularity category.}
Over the group algebra $kG$, there is an inclusion
$\underline{{\tt Cof}}(kG) \subseteq \underline{{\tt GProj}}(kG)$
of stable categories. Composing with the canonical functor
$\underline{{\tt GProj}}(kG)\rightarrow {\bf D}^b_{sg}(kG)$, we
obtain a triangulated functor
$F : \underline{{\tt Cof}}(kG) \rightarrow {\bf D}^b_{sg}(kG)$
which is fully faithful. Let  ${\bf K}^b({\tt Proj}(kG))$ be the
homotopy category of bounded complexes of projective $kG$-modules.
Since any projective module is cofibrant, it is clear that
${\bf K}^b({\tt Proj}(kG)) \subseteq {\bf D}^{b}(kG)_{f{\tt Cof}}$.

\begin{Lemma}\label{lem:tri-eq2}
There is a triangle-equivalence $\underline{{\tt Cof}}(kG) \simeq
{\bf D}^{b}(kG)_{f{\tt Cof}}/{\bf K}^b({\tt Proj}(kG))$.
\end{Lemma}

\begin{proof}
Consider the fully faithful functor
$F: \underline{{\tt Cof}}(kG)\rightarrow {\bf D}^b_{sg}(kG)$
defined above. For any cofibrant module $M$, it is clear that
$F(M)\in {\bf D}^{b}(kG)_{f{\tt Cof}}/{\bf K}^b({\tt Proj}(kG))
 \subseteq {\bf D}^b_{sg}(kG)$,
so that $F$ is actually a functor from $\underline{{\tt Cof}}(kG)$
to ${\bf D}^{b}(kG)_{f{\tt Cof}}/{\bf K}^b({\tt Proj}(kG))$. We shall
prove that the image of $F$ is dense in
${\bf D}^{b}(kG)_{f{\tt Cof}}/{\bf K}^b({\tt Proj}(kG))$. To this end,
let $X\in {\bf D}^{b}(kG)_{f{\tt Cof}}$ and note that there exists a
quasi-isomorphism $P \rightarrow X$ with
$P\in {\bf K}^{-, b}({\tt Proj}(kG))$, so that $P \simeq X$ in
${\bf D}^{b}(kG)_{f{\tt Cof}}$. Then, there exists an integer $n \ll 0$
such that ${\rm H}^i(P) = 0$ and ${\rm im}d^i$ is cofibrant for all
$i \leq n$.\footnote{Indeed, the mapping cone of the quasi-isomorphism
$P\rightarrow X$ is an acyclic complex in ${\bf K}^{-}({\tt Cof}(kG))$,
which coincides with $P$ at degrees $i \ll 0$. Since the class of cofibrant
modules is closed under kernels of epimorphisms, all image modules of the
mapping cone are cofibrant.} Since $P\in {\bf K}^{-, b}({\tt Proj}(kG))$,
the left brutal truncation $P^{\geq n+1}$ is contained in
${\bf K}^{b}({\tt Proj}(kG))$. There is a morphism
$P^{\leq n}[-1] \rightarrow P^{\geq n+1}$, which consists of $d_P^n$
in degree $n+1$ and $0$'s elsewhere, whose mapping cone is precisely the
complex $P$. Then, we obtain a distinguished triangle
\[ P^{\leq n}[-1]\longrightarrow P^{\geq n+1}\longrightarrow
   P\longrightarrow P^{\leq n}, \]
which implies that
$P \simeq P^{\leq n}$ in
${\bf D}^{b}(kG)_{f{\tt Cof}} / {\bf K}^b({\tt Proj}(kG))$. It follows that
$X \simeq P \simeq P^{\leq n} \simeq F({\rm im}d^n)$ and hence the image of
$F$ is dense in ${\bf D}^{b}(kG)_{f{\tt Cof}}/{\bf K}^b({\tt Proj}(kG))$.
\end{proof}

The following result shows that the cofibrant singularity category measures,
to some extent, the obstruction to the equality between Gorenstein projective
and cofibrant modules.

\begin{Theorem}\label{thm:Cof-sg2}
(i)
${\bf D}^b_{{\tt Cof}.sg}(kG) \simeq
 {\bf D}^b_{sg}(kG) / \underline{\tt Cof}(kG)$.

(ii) The group algebra $kG$ is Gorenstein if and only if there is
a triangle-equivalence
\[ {\bf D}^b_{{\tt Cof}.sg}(kG) \simeq
   \underline{{\tt GProj}}(kG) / \underline{{\tt Cof}}(kG). \]
\end{Theorem}

\begin{proof}
Since ${\bf D}^b_{sg}(kG)={\bf D}^b(kG)/{\bf K}^b({\tt Proj}(kG))$,
Lemma \ref{lem:tri-eq2}
implies that
\[ {\bf D}^b_{sg}(kG) /\underline{\tt Cof}(kG)\simeq {\bf D}^b(kG) / {\bf D}^b(kG)_{f{\tt Cof}}. \]
Then, (i) follows from Proposition \ref{prop:Cof-sg1}.

For (ii), we consider the following commutative diagram of
triangulated categories
\[ \xymatrix{
\underline{{\tt Cof}}(kG) \ar[d]_{\simeq} \ar@{^{(}->}[r]
& \underline{{\tt GProj}}(kG) \ar[r] \ar[d]_{}
& \underline{{\tt GProj}}(kG)/\underline{{\tt Cof}}(kG) \ar@{.>}[d]\\
{\bf D}^b(kG)_{f{\tt Cof}}/{\bf K}^b({\tt Proj}(kG)) \ar@{^{(}->}[r]
&{\bf D}^b(kG)/{\bf K}^b({\tt Proj}(kG))\ar[r] &
{\bf D}^b(kG) / {\bf D}^b(kG)_{f{\tt Cof}}} \]
where the first vertical equivalence holds by Lemma \ref{lem:tri-eq2}.
If the group algebra $kG$ is Gorenstein, then the second vertical functor
$\underline{{\tt GProj}}(kG) \rightarrow
 {\bf D}^b(kG)/{\bf K}^b({\tt Proj}) = {\bf D}^b_{sg}(kG)$
is an equivalence as well. Invoking Proposition \ref{prop:Cof-sg1},
we then obtain an equivalence
\[ {\bf D}^b_{{\tt Cof}.sg}(kG) \simeq
   {\bf D}^b(kG) / {\bf D}^b(kG)_{f{\tt Cof}} \simeq
   \underline{{\tt GProj}}(kG) / \underline{{\tt Cof}}(kG). \]
Conversely, if
${\bf D}^b_{sg}(kG) /\underline{\tt Cof}(kG) \simeq
 \underline{\tt GProj}(kG)/\underline{\tt Cof}(kG)$,
the triviality of the Verdier quotient
\[ ({\bf D}^b_{sg}(kG) / \underline{\tt Cof}(kG)) /
   (\underline{\tt GProj}(kG) / \underline{\tt Cof}(kG)) \simeq
   {\bf D}^b_{sg}(kG) / \underline{\tt GProj}(kG) \]
implies that ${\bf D}^b_{sg}(kG) \simeq \underline{{\tt GProj}}(kG)$,
so that $kG$ is Gorenstein.
\end{proof}

\noindent
Recall that the fully faithful functor $F$ from the stable category of
Gorenstein projective modules to ${\bf D}^b_{sg}(R)$ is an equivalence
if the ring $R$ has finite Gorenstein global
dimension; this result is referred to as Buchweitz's theorem
\cite[4.4.1]{Buc}. By applying the notion of the Gorenstein defect
category \cite[Definition 4.1]{BJO}, which is defined as the Verdier
quotient ${\bf D}^b_{sg}(R)/{\rm Im}F$, it follows from
\cite[Theorems 3.6 and 4.2]{BJO} that the converse of Buchweitz's
theorem is also true, so that the functor $F$ is an equivalence if
and only if $R$ is Gorenstein; see also \cite[\S 8.5]{Z}. Note that
the functor
$F : \underline{{\tt Cof}}(kG) \rightarrow {\bf D}^b_{sg}(kG)$,
obtained as the composition
\[ \underline{{\tt Cof}}(kG) \longrightarrow
   \underline{{\tt GProj}}(kG) \longrightarrow {\bf D}^b_{sg}(kG) , \]
is fully faithful. Regarding the density of the latter functor, we
have the following result. We note that a similar argument provides
a direct and concise proof for the converse of Buchweitz's theorem,
regarding Gorenstein projective modules; as far as we know, this
argument has not appeared before in the literature.

\begin{Proposition}\label{prop:dense}
If the functor
$F : \underline{{\tt Cof}}(kG) \rightarrow {\bf D}^b_{sg}(kG)$
is dense, then every $kG$-module has finite cofibrant dimension.
\end{Proposition}

\begin{proof}
Let $M$ be a $kG$-module. Since the functor $F$ is dense, there exists
a cofibrant module $M'$, such that $M=F(M')=M'$ in ${\bf D}^b_{sg}(kG)$.
We consider projective resolutions $P \rightarrow M$ and
$P' \rightarrow M'$ of $M$ and $M'$ respectively and note that
${\rm ker}d_{P'}^i$ is cofibrant for all $i \leq 0$. In the singularity
category ${\bf D}^b_{sg}(kG)$, we have $P = P'$. Since
${\bf D}^b(kG) \simeq {\bf K}^{-,b}({\tt Proj}(kG))$, that equality can
be represented by a right fraction
$P \stackrel{t}\Longleftarrow L\stackrel{t'}\longrightarrow P'$, where
$L$ is in ${\bf K}^{-,b}({\tt Proj}(kG))$, and both ${\rm Con}(t)$ and
${\rm Con}(t')$ are in ${\bf K}^{b}({\tt Proj}(kG))$. Since the class
of cofibrant modules contains all projective modules and is closed under
direct sums and direct summands, the following auxiliary result implies
that ${\rm ker}d_P^i$ is cofibrant for $i \ll 0$; hence, $M$ has finite
cofibrant dimension.
\end{proof}

\begin{Lemma}
Let $L,P$ be two complexes in ${\bf K}^{-,b}({\tt Proj}(kG))$ and
assume that $t : L \rightarrow P$ is a morphism whose mapping
cone is quasi-isomorphic with a complex in ${\bf K}^{b}({\tt Proj}(kG))$.
Then, the kernels ${\rm ker}d_L^i$ and ${\rm ker}d_P^i$ are stably
isomorphic for $i \ll 0$.
\end{Lemma}

\begin{proof}
We note that the mapping cone $K = {\rm Con}(t)$ is contained in
${\bf K}^{-,b}({\tt Proj}(kG))$. By our assumption, there is a
quasi-isomorphism between $K$ and a complex in
${\bf K}^{b}({\tt Proj}(kG))$. Let $C$ be the mapping cone of the
latter quasi-isomorphism. Then, $C$ is an acyclic complex in
${\bf K}^{-}({\tt Proj}(kG))$; as such, $C$ is contractible. Since
$C$ agrees with $K$ in degrees $\ll 0$, we conclude that the kernels
${\rm ker}d_K^i$ are projective for $i \ll 0$. Since $L$ is acyclic
in degrees $\ll 0$, the short exact sequence of complexes
\[ 0 \longrightarrow P \longrightarrow K \longrightarrow L[1]
     \longrightarrow 0 \]
induces short exact sequences of modules
\[ 0 \longrightarrow {\rm ker}d_P^i \longrightarrow {\rm ker}d_K^i
     \longrightarrow {\rm ker}d_L^{i+1} \longrightarrow 0 \]
for all $i \ll 0$. On the other hand, there are also short exact
sequences of modules
\[ 0 \longrightarrow {\rm ker}d_L^i \longrightarrow L^i
     \longrightarrow {\rm ker}d_L^{i+1} \longrightarrow 0 \]
for all $i \ll 0$. Then, the result follows from Schanuel's lemma.
\end{proof}

\begin{Theorem}\label{thm:Cof-sg=0}
Consider the following conditions:

(i) Every $kG$-module has finite cofibrant dimension.

(ii) ${\bf D}^b_{{\tt Cof}.sg}(kG) = 0$.

(iii) The canonical functor
$\underline{{\tt Cof}}(kG) \rightarrow {\bf D}^b_{sg}(kG)$
is an equivalence.

(iv) The canonical functor
${\bf K}^{b}({\tt Cof}(kG)) \rightarrow {\bf D}^{b}_{\tt Cof}(kG)$
is an equivalence.

(v) $\underline{{\tt Cof}}(kG) \simeq \underline{{\tt GProj}}(kG)$.\\
\noindent Then
(i)$\leftrightarrow$(ii)$\leftrightarrow$(iii)$\leftrightarrow$(iv)$\rightarrow$(v);
all conditions are equivalent if $kG$ is Gorenstein.
\end{Theorem}

\begin{proof}
All $kG$-modules have finite cofibrant dimension if and only if
${\bf D}^{b}(kG) = {\bf D}^{b}(kG)_{f{\tt Cof}}$. Invoking the equivalence
${\bf D}^b_{{\tt Cof}.sg}(kG)\simeq{\bf D}^{b}(kG)/{\bf D}^{b}(kG)_{f{\tt Cof}}$
in Proposition \ref{prop:Cof-sg1}, we conclude that (i)$\leftrightarrow$(ii).

(ii)$\rightarrow$(iii): Combining assertion (ii) with Proposition
\ref{prop:Cof-sg1}, we conclude that
${\bf D}^b(kG) = {\bf D}^b(kG)_{f{\tt Cof}}$. Then, Lemma \ref{lem:tri-eq2}
implies that
\[ \underline{{\tt Cof}}(kG) \simeq
   {\bf D}^{b}(kG)_{f{\tt Cof}} / {\bf K}^b({\tt Proj}(kG)) =
   {\bf D}^{b}(kG) / {\bf K}^b({\tt Proj}(kG)) =
   {\bf D}^b_{sg}(kG).\]

(iii)$\rightarrow$(i): This is proved in Proposition \ref{prop:dense}.

(i)$\rightarrow$(iv): Let $F$ be the composition of the inclusion
${\bf K}^{b}({\tt Cof}(kG)) \rightarrow {\bf K}^{b}(kG)$ and
the canonical localization functor
$Q : {\bf K}^{b}(kG) \rightarrow {\bf D}^{b}_{\tt Cof}(kG)$.
Lemma \ref{lem:iso1} implies that $F$ is fully faithful. Invoking
condition (i) on the finiteness of the cofibrant dimension, the
construction in Lemma \ref{lem:res} implies that the image of $F$
is dense. Hence, $F$ is an equivalence.

(iv)$\rightarrow$(i): Let $M$ be a $kG$-module and consider a cofibrant
resolution $C$ of $M$, obtained by taking successive special
${\tt Cof}(kG)$-precovers; then, $M=C$ in
${\bf D}^{b}_{\tt Cof}(kG) \simeq {\bf K}^{-,cab}({\tt Cof}(kG))$. In
view of assertion (iv), $C$ is homotopy equivalent to a complex
$C' \in {\bf K}^{b}({\tt Cof}(kG))$. The cone of a quasi-isomorphism
between $C$ and $C'$ is an acyclic complex of cofibrant modules in
${\bf K}^{-}({\tt Cof}(kG))$, which coincides with $C$ in degrees $\ll 0$.
Since the class of cofibrant modules is closed under kernels of epimorphisms,
all kernels of that cone are cofibrant. It follows that the kernels of $C$
are cofibrant in degrees $\ll 0$ and hence $M$ has finite cofibrant dimension.

(i)$\rightarrow$(v): In view of Lemma \ref{lem:bound-cd}, assertion
(i) implies that ${\tt Cof}(kG) = {\tt GProj}(kG)$, so that
$\underline{{\tt Cof}}(kG) = \underline{{\tt GProj}}(kG)$.

If we assume that $kG$ is Gorenstein and assertion (v) holds, then Theorem
\ref{thm:Cof-sg2}(ii) implies that ${\bf D}^b_{{\tt Cof}.sg}(kG) = 0$; this
proves that (v)$\rightarrow$(ii).
\end{proof}

\begin{Example}\label{eg:2}
(i) If $G$ is a finite group, then the group algebra $\mathbb{Z}G$
satisfies the equivalent conditions of Theorem 3.8. Indeed, in this
case, a $\mathbb{Z}G$-module is Gorenstein projective (cofibrant) if
and only if it is $\mathbb{Z}$-free. More generally, if $G$ is a group
of type $FP_\infty$ contained in the class
${\scriptstyle{{\bf LH}}}\mathfrak{F}$ of hierarchically decomposable
groups that were introduced by Kropholler in \cite{Kro}, then
$\mathbb{Z}G$ satisfies the above equivalent conditions as well; cf.
\cite{DT, ET2}.

(ii) Let $G$ be a finite group and $p$ a prime number. We consider
the local ring $k = \mathbb{Z}/(p^2)$; $k$ is self-injective and
the group algebra $kG$ is Gorenstein. Invoking Example \ref{eg:1}
and Theorem \ref{thm:Cof-sg=0}, it follows that the cofibrant singularity category
of $kG$ is not trivial.
\end{Example}

\section{Stable categories}

\noindent
In this section, we intend to study the relations between various stable
categories over the group algebra $kG$.
First, we need to recall the notion of cotorsion pairs.
Let ${\mathcal E}$ be an additive full and extension-closed
subcategory of the category of modules over a ring $R$; then,
$\mathcal{E}$ is an exact category in the sense of Quillen
\cite{Q}. The ${\rm Ext}^1$-pairing induces an orthogonality
relation between subclasses of ${\mathcal E}$.
Let ${\mathcal C},{\mathcal D}$ be two subclasses of ${\mathcal E}$.
Then, the pair $({\mathcal C},{\mathcal D})$ is a cotorsion pair in
${\mathcal E}$ (cf.\ \cite[Definition 7.1.2]{EJ}) if
${\mathcal C} = {^{\perp} {\mathcal D}}$ and
${\mathcal C} ^{\perp} = {\mathcal D}$. The cotorsion pair is
called hereditary if ${\rm Ext}^i_R(C,D)=0$ for all $i>0$ and
all modules $C \in \mathcal{C}$ and $D \in \mathcal{D}$. If
$\mathcal{E}$ has enough projective (resp.\ injective) objects,
the latter condition is equivalent to the assertion that
$\mathcal{C}$ (resp.\ $\mathcal{D}$) is closed under kernels
of epimorphisms (resp.\ under cokernels of monomorphisms). We
say that the cotorsion pair is complete if for any
$E \in {\mathcal E}$ there exist two short exact sequences of
modules
\[ 0 \longrightarrow D \longrightarrow C \longrightarrow E
     \longrightarrow 0
   \;\;\; \mbox{and } \;\;\;
   0 \longrightarrow E \longrightarrow D' \longrightarrow C'
     \longrightarrow 0 , \]
where $C,C' \in {\mathcal C}$ and $D,D' \in {\mathcal D}$.

\vspace{0.1in}

\noindent
{\sc I.\ A localization sequence.} Let $\mathcal{E}$ be a full and
extension-closed subcategory of the category of modules over a ring
$R$; then, $\mathcal{E}$ is an exact category in the sense of \cite{Q}.
We place ourselves in the following general setting: Let
$(\mathcal{A}, \mathcal{B})$ be a complete hereditary cotorsion pair
in $\mathcal{E}$. We consider a subclass $\mathcal{C}$ of $\mathcal{E}$,
which contains $\mathcal{A}$ and is closed under extensions and kernels
of epimorphisms, and a subclass $\mathcal{D}$ of $\mathcal{E}$, which
contains $\mathcal{B}$ and is closed under extensions and cokernels of
monomorphisms. We also assume that $\mathcal{A}\cap \mathcal{D}$,
$\mathcal{C}\cap \mathcal{D}$ and $\mathcal{C}\cap \mathcal{B}$ are
Frobenius categories with the same projective-injective objects, which
are precisely the modules in the kernel $\mathcal{A}\cap\mathcal{B}$
of the cotorsion pair $(\mathcal{A}, \mathcal{B})$. The stable categories
of the Frobenius categories $\mathcal{A} \cap \mathcal{D}$,
$\mathcal{C} \cap \mathcal{D}$ and $\mathcal{C}\cap \mathcal{B}$ are
denoted by $\underline{\mathcal{A}\cap \mathcal{D}}$,
$\underline{\mathcal{C}\cap \mathcal{D}}$ and
$\underline{\mathcal{C}\cap \mathcal{B}}$ respectively. The latter
categories are canonically triangulated categories;
cf.\ \cite[Theorem 2.6]{Ha}.

For any module $M \in \mathcal{C}\cap \mathcal{D}$ we consider a
short exact sequence in $\mathcal{E}$
\begin{equation}
 0 \longrightarrow B \longrightarrow A
   \stackrel{p}{\longrightarrow} M \longrightarrow 0,
\end{equation}
where $A\in \mathcal{A}$ and $B \in\mathcal{B}$. Our assumptions on
$\mathcal{C}$ and $\mathcal{D}$ imply that
$A \in \mathcal{A} \cap \mathcal{D}$ and
$B \in \mathcal{C} \cap \mathcal{B}$.

\begin{Lemma}\label{lem:maps}
Let $f : M \rightarrow M'$ be a morphism in $\mathcal{E}$,
where $M$ and $M'$ are objects in $\mathcal{C}\cap \mathcal{D}$.
We also consider two short exact sequences in $\mathcal{E}$:
\[ 0 \longrightarrow B \longrightarrow A
     \stackrel{p}{\longrightarrow} M \longrightarrow 0
   \;\;\; \mbox{and} \;\;\;
   0 \longrightarrow B' \stackrel{\iota'}{\longrightarrow} A'
     \stackrel{p'}{\longrightarrow} M' \longrightarrow 0 , \]
where $A,A' \in \mathcal{A}\cap \mathcal{D}$ and $B,B' \in
\mathcal{C}\cap \mathcal{B}$. Then:

(i) There exists a map $g : A \rightarrow A'$,
such that $p'g = fp$.

(ii) If $g,g' : A \rightarrow A'$ are two maps with
$p'g = fp = p'g'$, then
$[g] = [g'] \in \underline{\rm{Hom}}_R(A, A')$.

(iii) If $[f] = [0] \in \underline{\rm{Hom}}_R(M, M')$
and $g : A \rightarrow A'$ is a map with
$p'g=fp$, then $[g]=[0] \in \underline{\rm{Hom}}_R(A,A')$.
\end{Lemma}

\begin{proof}
(i) The additive map
$p'_* : \mbox{Hom}_R(A,A') \rightarrow \mbox{Hom}_R(A,M')$
is surjective, since the abelian group $\mbox{Ext}^1_R(A,B')$ is
trivial. Therefore, there exists a map $g : A \rightarrow A'$
such that $fp = p'_*(g) = p'g$, as needed.

(ii) Let $g,g' : A \rightarrow A'$ be two maps with
$p'g = fp = p'g'$.
\[
\begin{array}{ccccccccc}
 0 & \longrightarrow & B & \longrightarrow & A
   & \stackrel{p}{\longrightarrow} & M & \longrightarrow & 0 \\
 & & & & {\scriptstyle{g}} \downarrow \downarrow {\scriptstyle{g'}}
 & & \!\!\! {\scriptstyle{f}} \downarrow & & \\
 0 & \longrightarrow & B' & \stackrel{\imath'}{\longrightarrow}
   & A' & \stackrel{p'}{\longrightarrow} & M' & \longrightarrow & 0
\end{array}
\]
Then, $p'(g'-g) = p'g' - p'g = 0$ and hence there exists a map
$h : A \rightarrow B'$ with $g'-g = \iota'h$. We fix a surjective
map $\pi: P\rightarrow B'$, where $P$ is a projective-injective
object of $\mathcal{C} \cap \mathcal{B}$, and note that
${\rm ker}\pi\in \mathcal{C} \cap \mathcal{B} \subseteq \mathcal{B}$.
Then, $\mbox{Ext}^1_R(A, {\rm ker}\pi)=0$ and hence $h$ factors through
$\pi$. This is also the case for $g'-g = \iota'h$ and hence
$[g] = [g'] \in \underline{\mbox{Hom}}_R(A,A')$.

(iii) Assume that $f$ factors as the composition of two maps
$M \stackrel{a}{\rightarrow} Q \stackrel{b}{\rightarrow} M'$,
where $Q$ is a projective-injective object of $\mathcal{C}\cap \mathcal{D}$.
If $\beta : Q \rightarrow A'$ is a map with $p'\beta = b$, then
the composition $\beta ap: A \rightarrow A'$ is such that
$p'(\beta ap) = (p' \beta)ap = bap = fp$. It follows from (ii) above that
$[g] = [\beta ap] \in \underline{\mbox{Hom}}_R(A,A')$. This finishes
the proof, since we obviously have
$[\beta ap] = [0] \in \underline{\mbox{Hom}}_R(A,A')$.
\end{proof}

\noindent
Lemma \ref{lem:maps} implies that for any $M\in \mathcal{C}\cap \mathcal{D}$,
the module $A\in \mathcal{A}\cap \mathcal{D}$ that appears in the short exact
sequence (1) is uniquely determined by $M$, up to a canonical isomorphism in
the stable category $\underline{\mathcal{A}\cap \mathcal{D}}$. Moreover, Lemma
\ref{lem:maps}(iii) implies that the assignment $M \mapsto A$ factors through
the stable category $\underline{\mathcal{C}\cap \mathcal{D}}$ and defines a
functor
\[ i^! : \underline{\mathcal{C}\cap \mathcal{D}} \longrightarrow
         \underline{\mathcal{A}\cap \mathcal{D}}, \]
which is clearly additive.

\begin{Proposition}
The additive functor
$i^! : \underline{\mathcal{C} \cap \mathcal{D}} \longrightarrow
       \underline{\mathcal{A}\cap \mathcal{D}}$
defined above is right adjoint to the inclusion functor
$i_* : \underline{\mathcal{A} \cap \mathcal{D}} \longrightarrow
       \underline{\mathcal{C}\cap \mathcal{D}}$
and hence it is triangulated. In addition, the composition
$i^! \circ i_*$ is the identity on
$\underline{\mathcal{A} \cap \mathcal{D}}$.
\end{Proposition}

\begin{proof}
We fix $N\in \mathcal{A} \cap \mathcal{D}$ and let
$M \in \mathcal{C} \cap \mathcal{D}$. We also consider a short exact
sequence (1), where $B \in \mathcal{C} \cap \mathcal{B}$ and
$A \in \mathcal{A} \cap \mathcal{D}$. We note that the additive map
\[ [p]_{*} : \underline{\mbox{Hom}}_R(N,A) \longrightarrow
           \underline{\mbox{Hom}}_R(N,M) \]
is natural in both $N$ (this is obvious) and $M$ (this follows from
Lemma \ref{lem:maps}(ii)). We establish the adjunction in the statement
of the Proposition, by proving that $[p]_{*}$ is bijective. Indeed,
since the group $\mbox{Ext}^{1}_R(N,B)$ is trivial, the additive map
\[ p_{*} : \mbox{Hom}_R(N,A) \longrightarrow \mbox{Hom}_R(N,M) \]
is surjective, whence the surjectivity of $[p]_{*}$. Regarding the
injectivity of $[p]_{*}$, consider a map $f : N \rightarrow A$,
such that
$[pf] = [p] \cdot [f] = [p]_{*}[f] = [0] \in
 \underline{\mbox{Hom}}_R(N,M)$.
Then, we may consider the commutative diagram
\[
\begin{array}{ccccccccc}
 0 & \longrightarrow & 0 & \longrightarrow & N
   & \stackrel{1_N}{\longrightarrow} & N
   & \longrightarrow & 0 \\
 & & & & \!\!\! {\scriptstyle{f}} \downarrow & &
         \!\!\!\!\! {\scriptstyle{pf}} \downarrow
 & & \\
 0 & \longrightarrow & B & \longrightarrow & A
   & \stackrel{p}{\longrightarrow} & M & \longrightarrow & 0
\end{array}
\]
and invoke Lemma \ref{lem:maps}(iii), in order to conclude that
$[f] = [0] \in \underline{\mbox{Hom}}_R(N,A)$.

Being right adjoint to the triangulated functor $i_*$, the
functor $i^!$ is also triangulated; cf.\ \cite[Lemma 5.3.6]{N1}.
In order to verify that the composition $i^! \circ i_*$ is
the identity on $\underline{\mathcal{A} \cap \mathcal{D}}$, we
simply note that for any $M \in \mathcal{A} \cap \mathcal{D}$
we can choose the approximation sequence
\[ 0 \longrightarrow 0 \longrightarrow M
     \stackrel{1_M}{\longrightarrow} M \longrightarrow 0 , \]
so that $i^!M=M$.
\end{proof}

\vspace{0.1in}

\noindent
The completeness of the cotorsion pair
$\left( \mathcal{A}, \mathcal{B} \right)$ also implies that
for any $M\in \mathcal{C}\cap \mathcal{D}$ there exists a
short exact sequence in $\mathcal{E}$
\begin{equation}
 0 \longrightarrow M \longrightarrow B \longrightarrow A
   \longrightarrow 0 ,
\end{equation}
where $B \in \mathcal{B}$ and $A \in \mathcal{A}$. Our assumptions
on $\mathcal{C}$ and $\mathcal{D}$ imply that
$A\in \mathcal{A}\cap \mathcal{D}$ and $B\in \mathcal{C}\cap \mathcal{B}$.
Working as above, we can show that $B$ is uniquely determined, up to a
canonical isomorphism in the stable category
$\underline{\mathcal{C}\cap \mathcal{B}}$, by $M$
and the assignment $M \mapsto B$ defines an additive functor
\[ j^* : \underline{\mathcal{C} \cap \mathcal{D}} \longrightarrow
         \underline{\mathcal{C} \cap \mathcal{B}}, \]
which is left adjoint to the inclusion functor
$j_* : \underline{\mathcal{C} \cap \mathcal{B}} \longrightarrow
       \underline{\mathcal{C}\cap \mathcal{D}}$.
In particular, the functor $j^*$ is triangulated. Moreover,
the composition $j^* \circ j_*$ is the identity on
$\underline{\mathcal{C}\cap \mathcal{B}}$.

\begin{Lemma}
For any $M \in \mathcal{C}\cap \mathcal{D}$, we have
$j^*M = 0 \in \underline{\mathcal{C} \cap \mathcal{B}}$
if and only if $M \in \mathcal{A} \cap \mathcal{D}$.
\end{Lemma}

\begin{proof}
The short exact sequence (2) implies that
$M \in \mathcal{A} \cap \mathcal{D}$ if and only if
$B \in \mathcal{A}\cap \mathcal{B}$. On the other hand,
$B = j^*M = 0 \in \underline{\mathcal{C} \cap \mathcal{B}}$
if and only if $B$ is a projective-injective object of the
Frobenius category $\mathcal{C} \cap \mathcal{B}$. This
completes the proof, since the projective-injective objects
of $\mathcal{C} \cap \mathcal{B}$ are precisely the modules
in $\mathcal{A} \cap \mathcal{B}$.
\end{proof}

\noindent
We summarize the discussion above in the form of the following
result, which establishes the existence of a localization sequence
of triangulated categories; cf.\ \cite{N2, V}.

\begin{Theorem}\label{thm:loc-seq}
The functors defined above induce a localization sequence
\[ \underline{\mathcal{A}\cap \mathcal{D}}
   \stackrel{i_*}{\longrightarrow}
   \underline{\mathcal{C} \cap\mathcal{D}}
   \stackrel{j^*}{\longrightarrow}
   \underline{\mathcal{C} \cap \mathcal{B}}. \]
The right adjoint of the inclusion $i_*$ is
$i^! : \underline{\mathcal{C} \cap\mathcal{D}}
       \longrightarrow
       \underline{\mathcal{A} \cap \mathcal{D}}$
and the right adjoint of $j^*$ is the inclusion
$j_* : \underline{\mathcal{C} \cap \mathcal{B}}
       \longrightarrow
       \underline{\mathcal{C} \cap\mathcal{D}}$.
Consequently, the functor $j^*$ induces an
equivalence of triangulated categories
\[ \underline{\mathcal{C} \cap\mathcal{D}} /
   \underline{\mathcal{A} \cap\mathcal{D}}
   \stackrel{\sim}{\longrightarrow}
   \underline{\mathcal{C} \cap \mathcal{B}}. \]
\end{Theorem}

\vspace{0.1in}
\noindent
{\sc II.\ Stable categories of $kG$-modules.}
We now specialize the discussion in Subsection I above to the case of
certain classes of modules over the group algebra $kG$ of a group $G$.

Let ${\tt PGF}(kG)$ be the class of projectively coresolved
Gorenstein flat $kG$-modules. It is easily seen that ${\tt PGF}(kG)$
is a Frobenius category with projective-injective objects given by
the projective modules. It follows from \cite[Lemma 4.5]{ER1} that
${\tt Cof}(kG)$ is also a Frobenius category with projective-injective
objects given by the projective modules. Finally, using the arguments
in the proof of \cite[Lemma 4.5]{ER1}, it follows that the category
${\tt PGF}(kG)\cap {\tt Cof}(kG)^{\perp}$ is Frobenius, with
projective-injective objects the projective modules as well.

\begin{Proposition}\label{prop:cof-PGF}
There is a localization sequence
\[ \underline{\tt Cof}(kG)\stackrel{i_*}{\longrightarrow}
   \underline{\tt PGF}(kG) \stackrel{j^*}{\longrightarrow}
   \underline{{\tt PGF}\cap {\tt Cof}^\perp}(kG). \]
The functor $j^*$ induces a triangle-equivalence
\[ \underline{\tt PGF}(kG)/ \underline{\tt Cof}(kG)\simeq
   \underline{{\tt PGF}\cap {\tt Cof}^\perp}(kG). \]
\end{Proposition}

\begin{proof}
We consider the complete hereditary cotorsion pair
$({\tt Cof}(kG), {\tt Cof}(kG)^\perp)$, whose kernel is the
class of projective $kG$-modules; cf.\ \cite[Theorem 3.3]{ER1}.
As shown in \cite[Proposition 8.2]{S}, we have an inclusion
${\tt Cof}(kG) \subseteq {\tt PGF}(kG)$. We note that the class
${\tt PGF}(kG)$ is closed under extensions and kernels of
epimorphisms; cf.\ \cite[Theorem 4.9]{SS}). Letting
$\mathcal{E} = \mathcal{D}$ be the class of all $kG$-modules and
$\mathcal{C}$ the class of projectively coresolved Gorenstein flat
modules, the discussion above shows that all of the hypotheses in
the beginning of $\S $4.I are satisfied. The result is therefore
a particular case of Theorem \ref{thm:loc-seq}.
\end{proof}

\begin{Remark}
The question as to whether all Gorenstein projective modules
are projectively coresolved Gorenstein flat is open. If $R$
is any ring, then $({\tt PGF}(R), {\tt PGF}(R)^\perp)$ is a
complete hereditary cotorsion pair; this is proved in
\cite[Theorem 4.9]{SS}. Hence, as yet another application of
Theorem \ref{thm:loc-seq}, we obtain a localization sequence
\[ \underline{\tt PGF}(R)\stackrel{i_*}{\longrightarrow}
   \underline{\tt GProj}(R) \stackrel{j^*}{\longrightarrow}
   \underline{{\tt GProj} \cap {\tt PGF}^\perp}(R) , \]
with the functor $j^*$ inducing a triangle-equivalence
\[ \underline{\tt GProj}(R)/ \underline{\tt PGF}(R)\simeq
   \underline{{\tt GProj} \cap {\tt PGF}^\perp}(R). \]
It follows that all Gorenstein projective modules are
projectively coresolved Gorenstein flat if and only if
${\tt GProj}(R)\cap{\tt PGF}(R)^\perp$ is the class of
projective modules.
\end{Remark}

\noindent
Recall that a $kG$-module $C$ is cotorsion if ${\rm Ext}^1_{kG}(F, C)=0$
for any flat $kG$-module $F$. Since any flat $kG$-module is cofibrant-flat,
it follows that ${\tt Cof.flat}(kG)^{\perp} \subseteq {\tt Cotor}(kG)$. The
category ${\tt GFlat}(kG)\cap {\tt Cotor}(kG)$ of cotorsion Gorenstein flat
modules is Frobenius, with projective-injective objects the flat cotorsion
modules; cf.\ \cite[Theorem 5.6]{E2}. Hence, any cotorsion Gorenstein flat
module is a cokernel of an acyclic complex of flat cotorsion $kG$-modules,
all of whose cokernels are cotorsion Gorenstein flat.

\begin{Lemma}
The categories ${\tt Cof.flat}(kG)\cap {\tt Cotor}(kG)$ and
${\tt GFlat}(kG) \cap {\tt Cof.flat}(kG)^{\perp}$ are Frobenius with
projective-injective objects (in both cases) the flat cotorsion modules.
\end{Lemma}

\begin{proof}
The argument regarding ${\tt Cof.flat}(kG)\cap {\tt Cotor}(kG)$ is
completely analogous to that in the proof of \cite[Theorem 5.6]{E2}.
Regarding ${\tt GFlat}(kG) \cap {\tt Cof.flat}(kG)^{\perp}$, we note
that any $kG$-module $M$ contained therein is a cokernel of an acyclic
complex of flat cotorsion $kG$-modules, all of whose cokernels are also
contained in ${\tt GFlat}(kG) \cap {\tt Cof.flat}(kG)^{\perp}$. Indeed,
let $F$ be an acyclic complex of flat cotorsion $kG$-modules with $M=C_0F$,
all of whose cokernels are cotorsion Gorenstein flat. Then, these cokernels
are contained in ${\tt GFlat}(kG) \cap {\tt Cof.flat}(kG)^{\perp}$, since
the class ${\tt Cof.flat}(kG)^{\perp}$ contains all flat cotorsion
$kG$-modules and has the 2-out-of-3 property for short exact sequences
of cotorsion modules; cf. \cite[Proposition 2.9]{ER2}. Then, the standard
argument in the proof of \cite[Theorem 5.6]{E2} shows that
${\tt GFlat}(kG) \cap {\tt Cof.flat}(kG)^{\perp}$ is also Frobenius with
projective-injective objects the flat cotorsion modules.
\end{proof}

\begin{Proposition}\label{prop:CF-GF}
There is a localization sequence
\[ \underline{{\tt Cof.flat}\cap{\tt Cotor}}(kG)\stackrel{i_*}{\longrightarrow}
   \underline{{\tt GFlat}\cap{\tt Cotor}}(kG) \stackrel{j^*}{\longrightarrow}
   \underline{{\tt GFlat}\cap {\tt Cof.flat}^\perp}(kG). \]
The functor $j^*$ induces a triangle-equivalence
\[ \underline{{\tt GFlat}\cap{\tt Cotor}}(kG)/ \underline{{\tt Cof.flat}\cap{\tt Cotor}}(kG)
\simeq \underline{{\tt GFlat}\cap {\tt Cof.flat}^\perp}(kG). \]
\end{Proposition}

\begin{proof}
It follows from \cite[Theorem 2.8]{ER2} that
$\left( {\tt Cof.flat}(kG),{\tt Cof.flat}(kG)^{\perp} \right)$ is a
complete hereditary cotorsion pair, whose kernel is the class of flat
cotorsion modules. We also note that the class ${\tt GFlat}(kG)$ is
closed under extensions and kernels of epimorphisms; cf.\ \cite[Corollary
4.12]{SS}. Letting $\mathcal{E}$ be the class of all $kG$-modules,
$\mathcal{C}$ the class of Gorenstein flat modules and $\mathcal{D}$
the class of cotorsion modules, the discussion above shows that all of
the hypotheses in the beginning of $\S $4.I are satisfied. The result
is therefore a particular case of Theorem \ref{thm:loc-seq}.
\end{proof}

\begin{Corollary}
The following conditions are equivalent for the group algebra $kG$:

(i) ${\tt PGF}(kG) = {\tt Cof}(kG)$.

(ii) ${\tt GFlat}(kG) = {\tt Cof.flat}(kG)$.

(iii) ${\tt GFlat}(kG)\cap{\tt Cotor}(kG) = {\tt Cof.flat}(kG)\cap {\tt Cotor}(kG)$.

(iv) ${\tt PGF}(kG) \cap {\tt Cof}(kG)^\perp = {\tt Proj}(kG)$.

(v) ${\tt GFlat}(kG)\cap {\tt Cof.flat}(kG)^\perp = {\tt Flat}(kG) \cap {\tt Cotor}(kG)$.
\end{Corollary}

\begin{proof}
The equivalence (i)$\leftrightarrow$(ii) is precisely
\cite[Proposition 3.6]{ER2} and the implication
(ii)$\rightarrow$(iii) is clear. Since
${\tt Cof}(kG) \cap {\tt Cof}(kG)^\perp = {\tt Proj}(kG)$, it
follows that (i)$\rightarrow$(iv). The implication (iii)$\rightarrow$(v)
follows since ${\tt Cof.flat}(kG)^{\perp} \! \, \subseteq {\tt Cotor}(kG)$
and
${\tt Cof.flat}(kG) \cap {\tt Cof.flat}(kG)^\perp =
 {\tt Flat}(kG) \cap {\tt Cotor}(kG)$.

(iv)$\rightarrow$(i): Assumption (iv) implies that the stable category
$\underline{{\tt PGF}\cap{\tt Cof}^\perp}(kG)$ is trivial. Invoking
Proposition \ref{prop:cof-PGF}, we then conclude that
$\underline{\tt PGF}(kG) = \underline{\tt Cof}(kG)$. Since projective
modules are cofibrant, the closure of ${\tt Cof}(kG)$ under direct sums
and direct summands implies that ${\tt PGF}(kG) \subseteq {\tt Cof}(kG)$,
so that ${\tt PGF}(kG) = {\tt Cof}(kG)$.

(v)$\rightarrow$(iii): We may use Proposition \ref{prop:CF-GF} and work
as in the proof of the implication (iv)$\rightarrow$(i).

(iii)$\rightarrow$(ii): Let $M \in {\tt GFlat}(kG)$ and consider a
short exact sequence of $kG$-modules
\[ 0 \longrightarrow M\longrightarrow C\longrightarrow F\longrightarrow 0,\]
where $C$ is cotorsion and $F$ is flat. Then, $F$ is Gorenstein flat and
hence $C$ is also Gorenstein flat. Assumption (iii) implies that $C$ is
cofibrant-flat. Since $F$ is also cofibrant-flat, the closure of
${\tt Cof.flat}(kG)$ under kernels of epimorphisms implies that
$M \in {\tt Cof.flat}(kG)$.
\end{proof}

\noindent
A classical result, due to Bass \cite{Bass}, states that a ring is
perfect if and only if every flat module is projective. The following
result is a version of Bass' theorem, that involves the homological
behaviour of Gorenstein modules over group algebras.

\begin{Proposition}\label{prop:perf-ring}
The following conditions are equivalent for the group algebra $kG$:

(i) $kG$ is a perfect ring.

(ii) ${\tt Cof}(kG) = {\tt Cof.flat}(kG)$.

(iii) ${\tt PGF}(kG)= {\tt GFlat}(kG)$.

(iv) ${\tt Flat}(kG)\cap{\tt Cotor}(kG)= {\tt Proj}(kG)$.

(v) ${\tt Cof.flat}(kG)\cap{\tt Cotor}(kG)= {\tt Cof}(kG)$.
\end{Proposition}

\begin{proof}
If $kG$ is perfect, then all flat $kG$-modules are projective and
every $kG$-module is cotorsion. It follows that
${\tt Cof.flat}(kG) \cap {\tt Cotor}(kG) = {\tt Cof.flat}(kG) =
 {\tt Cof}(kG)$
and hence it is clear that (i)$\rightarrow$(ii),(iii),(iv),(v).

(iv)$\rightarrow$(i): Equality (iv) implies that all projective
$kG$-modules are cotorsion. Then, the ring $kG$ is perfect, in
view of \cite[Corollary 10]{GAH}.

Since cofibrant modules are projectively coresolved Gorenstein
flat, \cite[Theorem 4.4]{SS} implies that
${\tt Proj}(kG) = {\tt Cof}(kG) \cap {\tt Flat}(kG) =
 {\tt PGF}(kG) \cap {\tt Flat}(kG)$.
Hence, the implications (ii)$\rightarrow$(i), (iii)$\rightarrow$(i)
and (v)$\rightarrow$(iv) follow by taking the intersection of the
respective equalities with ${\tt Flat}(kG)$.
\end{proof}

\noindent
{\sc III.\ A colocalization sequence.}
For the class ${\tt Fib}(kG)$ of fibrant modules, it follows
from \cite[Theorem 4.5 and Proposition 4.6]{ER2} that
$({^\perp{\tt Fib}}(kG), {\tt Fib}(kG))$ is a complete hereditary
cotorsion pair. On the other hand, it is standard to show that
${\tt GInj}(kG)$, ${\tt Fib}(kG)$ and
${\tt GInj}(kG)\cap {^\perp{\tt Fib}}(kG)$ are Frobenius categories
with projective-injective objects the injective $kG$-modules.

For any $M \in {\tt GInj}(kG)$ there exists a short exact
sequence of $kG$-modules
\[ 0 \longrightarrow M \longrightarrow N \longrightarrow L
     \longrightarrow 0 , \]
where $N \in {\tt Fib}(kG)$ and
$L \in {\tt GInj}(kG) \cap {^\perp}{\tt Fib}(kG)$.
Working as in Lemma \ref{lem:maps}, we can show that $N$ is
uniquely determined by $M$, up to a canonical isomorphism in
the stable category $\underline{\tt Fib}(kG)$, and the
assignment $M \mapsto N$ defines a functor
\[ i^* : \underline{\tt GInj}(kG) \longrightarrow
         \underline{\tt Fib}(kG), \]
which is clearly additive. In fact,
$i^* : \underline{\tt GInj}(kG) \longrightarrow
       \underline{\tt Fib}(kG)$
is left adjoint to the inclusion functor
$i_* : \underline{\tt Fib}(kG) \longrightarrow
       \underline{\tt GInj}(kG)$
and the composition $i^* \circ i_*$ is the identity on
$\underline{\tt Fib}(kG)$.

Similarly, for any $M \in {\tt GInj}(kG)$ there exists
another short exact sequence of $kG$-modules
\[ 0 \longrightarrow K \longrightarrow J \longrightarrow M
     \longrightarrow 0 , \]
where $K \in {\tt Fib}(kG)$ and
$J \in {\tt GInj}(kG) \cap {^\perp}{\tt Fib}(kG)$. Then, $J$
is uniquely determined by $M$, up to a canonical isomorphism
in the stable category
$\underline{{\tt GInj} \cap {^\perp}{\tt Fib}}(kG)$, and the
assignment $M \mapsto J$ defines an additive functor
\[ j^*: \underline{\tt GInj}(kG) \longrightarrow
        \underline{{\tt GInj} \cap {^\perp}{\tt Fib}}(kG) , \]
which is right adjoint to the inclusion
$j_!: \underline{{\tt GInj} \cap {^\perp}{\tt Fib}}(kG)
\longrightarrow \underline{\tt GInj}(kG)$. The composition
$j^* \circ j_!$ is the identity on
$\underline{{\tt GInj} \cap {^\perp}{\tt Fib}}(kG)$ and
$j^*M=0 \in \underline{{\tt GInj} \cap {^\perp}{\tt Fib}}(kG)$
if and only if $M$ is fibrant.

We may summarize the discussion above in the form of the
following result, which establishes the existence of a
colocalization sequence of triangulated categories; cf.\ \cite{N2, V}.

\begin{Theorem}\label{thm:col-seq}
(i) The functors defined above induce a colocalization sequence
\[ \underline{\tt Fib}(kG) \stackrel{i_*}{\longrightarrow}
   \underline{\tt GInj}(kG) \stackrel{j^*}{\longrightarrow}
   \underline{{\tt GInj }\cap{^\perp}{\tt Fib}}(kG). \]
The left adjoint of the inclusion $i_*$ is
$i^* : \underline{\tt GInj}(kG)\longrightarrow
       \underline{\tt Fib}(kG)$
and the left adjoint of $j^*$ is the inclusion
$j_! : \underline{{\tt GInj} \cap {^\perp}{\tt Fib}}(kG)
       \longrightarrow \underline{\tt GInj}(kG)$.

(ii) The functor $j^*$ induces an equivalence of triangulated
categories
\[ \underline{\tt GInj}(kG)/\underline{\tt Fib}(kG)
   \stackrel{\sim}{\longrightarrow}
   \underline{{\tt GInj} \cap {^\perp}{\tt Fib}}(kG). \]

(iii) A Gorenstein injective $kG$-module $M$ is fibrant if
and only if ${\underline{\rm Hom}}_{kG}(L, M)=0$ for any
$L \in {\tt GInj}(kG) \cap {^\perp}{\tt Fib}(kG)$.
\end{Theorem}

\begin{Corollary}
The following conditions are equivalent for the group algebra
$kG$:

(i) ${\tt Fib}(kG) = {\tt GInj}(kG)$,

(ii) ${\tt GInj}(kG) \cap {\tt {^\perp}Fib}(kG) =
      {\tt Inj}(kG)$,

(iii) ${\tt GInj}(kG) \cap {\tt {^\perp}Fib}(kG)
       \subseteq {\tt Fib}(kG)$,

(iv) ${\tt GInj}(kG) \cap {\tt {^\perp}Fib}(kG)
      \subseteq {^\perp}{\tt GInj}(kG)$.
\end{Corollary}

\begin{proof}
We note that ${\tt Fib}(kG)={\tt GInj}(kG)$ if and only if
$\underline{{\tt Fib}}(kG) = \underline{{\tt GInj}}(kG)$;
this follows since injective modules are fibrant, whereas
${\tt Fib}(kG)$ is closed under finite direct sums and direct
summands. Hence, the equivalence (i)$\leftrightarrow$(ii)
follows from Theorem \ref{thm:col-seq}(ii).

It is clear that (ii)$\rightarrow$(iii), (iv). The implications
(iii)$\rightarrow$(ii) and (iv)$\rightarrow$(ii) follow easily,
since ${\tt Fib}(kG) \cap {^\perp}{\tt Fib}(kG) = {\tt Inj}(kG)$
and ${\tt GInj}(kG) \cap {^\perp}{\tt GInj}(kG) = {\tt Inj}(kG)$.
\end{proof}

\vspace{0.1in}

\noindent
{\bf Author Contributions}
I. Emmanouil and W. Ren have contributed equally to this work.
\vspace{0.1in}

\noindent
{\bf Funding}
I.\ Emmanouil was supported by the Hellenic Foundation for
Research and Innovation (H.F.R.I.) under the "3rd Call for
H.F.R.I.\ Research Projects to Support Faculty Members and
Researchers", project number 24921. W.\ Ren is supported by
the Natural Science Foundation of Chongqing, China (No.\
CSTB2025NSCQ-GPX1014).
\vspace{0.1in}

\noindent
{\bf Data Availability}	No datasets	were generated or analysed during the current study.
\vspace{0.2in}

\noindent
{\bf Declarations}

\noindent
{\bf Conflict of interest} The authors declare no conflict of interest.

\vspace{0.1in}

{\footnotesize \noindent Ioannis Emmanouil\\
Department of Mathematics, University of Athens, Athens 15784, Greece\\
E-mail: {\tt emmanoui$\symbol{64}$math.uoa.gr}}

\vspace{0.05in}

{\footnotesize \noindent Wei Ren\\
 School of Mathematical Sciences, Chongqing Normal University, Chongqing 401331, PR China\\
 E-mail: {\tt wren$\symbol{64}$cqnu.edu.cn}}

\end{document}